\documentclass{amsart}
\usepackage{amsmath,amssymb,amsfonts, amsthm}
\usepackage{stmaryrd}
\usepackage{bbold}
\usepackage{mathtools}
\usepackage{thmtools}
\usepackage{comment}
\usepackage{float}
\usepackage{xcolor}
\definecolor{purple}{rgb}{0.53,0.0,0.69} 
\definecolor{darkblue}{rgb}{0,0,0.6} 
\usepackage[pdfborder=0,pagebackref,colorlinks,citecolor=purple,linkcolor=purple,urlcolor=darkblue]{hyperref}
\usepackage{mathpartir}
\usepackage[all,2cell]{xy}\UseAllTwocells
\usepackage{tikz-cd}
\usepackage{color}
\usepackage{enumitem}
\usepackage{cleveref}
\crefname{thm}{Theorem}{Theorems}
\crefname{cor}{Corollary}{Corollaries}
\crefname{prop}{Proposition}{Propositions}
\crefname{defn}{Definition}{Definitions}
\crefname{rmk}{Remark}{Remarks}
\crefname{fig}{Figure}{Figures}
\crefname{ex}{Example}{Examples}
\crefname{lem}{Lemma}{Lemmas}
\crefname{ax}{Axiom}{Axioms}

\crefname{figure}{Figure}{Figures}

\ifx\color@rgb\@undefined\else
  \definecolor{maroon}{rgb}{0.5,0,0}
\fi

\theoremstyle{plain}
\newtheorem{thm}{Theorem}[section] 
\newtheorem{cor}[thm]{Corollary}
\newtheorem{prop}[thm]{Proposition}
\newtheorem{lem}[thm]{Lemma}
\newtheorem{conj}[thm]{Conjecture}
\newtheorem{cons}[thm]{Construction}
\newtheorem{var}[thm]{Variant}

\theoremstyle{definition}
\newtheorem{defn}[thm]{Definition}
\newtheorem{ntn}[thm]{Notation}

\theoremstyle{remark}
\newtheorem{ex}[thm]{Example}
\newtheorem{rmk}[thm]{Remark}
\newtheorem{war}[thm]{Warning}

\makeatletter
\let\c@equation\c@thm
\makeatother
\numberwithin{equation}{section}

\newcommand{\Cat}{\mathsf{Cat}}
\newcommand{\Set}{\mathsf{Set}}

\renewcommand{\S}{\mathcal{S}}
\renewcommand{\i}{\infty}
\renewcommand{\lim}{\varprojlim}
\renewcommand{\P}{Poly_{\mathcal{S}}}
\newcommand{\colim}{\varinjlim}

\title{On polynomial functors and polynomial comonads over infinity groupoids}
\author{Kun Chen}

\address{Shanghai center for mathematical sciences\\ Fudan University\\ 2005 Songhu Road 200438\\ Shanghai, China}
\email{chenkun21@m.fudan.edu.cn}
\begin{document}
\maketitle

\begin{abstract}
  We show that single-variable polynomial functors over the category $\S$ of infinity groupoids, as defined by Gepner-Haugseng-Kock in \cite{gepner2022operads}, are exactly colimits of representable copresheaves indexed by infinity groupoid. This allows us to establish certain categorical properties of the $\i$-category $\P$, in parallel with the case of the ordinary category $Poly$. We define the notion of polynomial comonad under the monoidal structure of $\P$ induced by composition of polynomials, and describe a construction toward exploring the connection between polynomial comonads and complete Segal spaces. This construction partially generalizes the classical one given in the proof of a theorem of Ahman-Uustalu.
\end{abstract}

\tableofcontents

\section{Introduction}
\label{sec:introduction}
\subsection{Background and motivation}
In ordinary category theory, the notion of \textit{polynomial functor} is a categorification of the usual notion of polynomial in algebra. In applied category theory, the category of single-variable polynomial functors over sets, denoted as $Poly$, has found remarkable applications to e.g. discrete dynamic systems (\cite{dnmc}), database theory (\cite{spivak2025functorial}), information theory (\cite{shannon}) and artificial intelligence (\cite{lang}). The category $Poly$ possesses various formal properties that are interesting from a computational point of view. For a standard reference on $Poly$, we refer to the book \cite{niu2023polynomial} or the survey \cite{handbook}. To give a definition, we typically write a polynomial functor $p\in Poly$ as a coproduct\begin{equation}\label{1steqt}
    p=\sum_{b\in B}y^{p[b]}.
\end{equation}Here $p[b]$ is a family of sets, parametrized by another set $B$. The variable $y$ is thought as a set instead of a number in the usual notion of polynomial, and $y^{p[b]}$ is understood as an endofunctor over $\Set$, which is the coYoneda embedding of the set $p[b]$.
\begin{war}
    Polynomial functor does not exactly correspond to the usual notion of polynomial like how the category of finite sets categorifies $\mathbb{N}$, since we allow e.g. $p=\mathbb{Q}y^{\mathbb{R}}$ as a polynomial functor. Even when the sets $p[b]$ and $B$ are all finite, this corresponds to polynomials with natural number coefficients.
\end{war}
In view of the rich theory discovered for the ordinary category $Poly$, it is reasonable to expect that an infinity categorical generalization might be of interests. In a such generalization, the fundamental role of $\Set$ is replaced by the $\i$-category $\S$ of infinity groupoids. The goal of this paper is to take a first step to study the $\i$-category $\P$ of polynomial functors over infinity groupoids, and to transport some known results of $Poly$ into the world of $\P$.\par
Related to our work, in \cite[section 2]{gepner2022operads}, the authors gave a definition of (multi-variable) polynomial functor over $\S$, which we will review in \cref{reviewdef}. However, for the motivation of our work, it is more desirable to have a definition mimicking \cref{1steqt}. There recall that:\par
\textit{A polynomial functor over $\Set$ is a coproduct of representable copresheaves.}\par
If we interpret coproduct as a colimit over discrete diagram, the data of a such polynomial functor can be identified with a functor $p:B\rightarrow\Set$ where $B\in\Set$. Inspired by this observation, it is tempting to give an alternative definition:\par
\textit{A polynomial functor over $\S$ is a colimit of representable copresheaves indexed by $\i$-groupoid.}\par
In fact, our first main result shows that our definition indeed coincides with that of Gepner-Haugseng-Kock. Approaching the theory of $\P$ based on this definition above, in this paper we are able to derive some similar results as in the case of $Poly$. We show that the $\i$-category $\P$ is bicomplete. We give an explicit formula for the composition of two polynomials as a polynomial, and show that $\P$ has a closed monoidal infinity structure induced by the composition product. Moreover, we want to study comonoid objects (whose definition in general will be given in \cref{comoddef}) in this monoidal structure. This is motivated by a beautiful theorem on $Poly$, originally discovered by Ahman and Uustalu in the study of computer science (\cite{mainthm}), which roughly states that:\par
\textit{Polynomial comonads over $\Set$ are the same as (small) categories.}\par
For a precise statement and an overview, see \cref{lablesubsec}. More generally, in \cite{neizaifanchou} it is proved that polynomial comonads over a general locally Cartesian closed category $\mathcal{E}$ are exactly the internal category objects in $\mathcal{E}$. Following this line of thought, we might speculate that comonoid objects in $\P$ exactly correspond to the internal $\i$-categories in $\S$, which are known to be \textbf{complete Segal spaces} by \cite[corollary 4.3.16]{arxivagain}. That is:\par
\textit{Polynomial comonads over $\S$ are the same as complete Segal spaces.}\par
The precise statement is still absent currently, but we will formulate a conjecture in \cref{mainconj}, and discuss a partial result toward it. More specifically, in \cref{lastsubsec} we prove that every comonoid object in $\P$ canonically gives rise to an augmented cosimplicial space satisfying a Segal-like condition. We mention that the main ingredient in this construction, is a functor (\cref{maincons}) that is conspicuously similar to the procedure of \textit{d{\'e}calage} of simplicial objects.\par
This paper is organized as follows. Section 2 is preparatory, in which we collect some technical lemmas that are useful for the later development. In section 3 we introduce our main formalization of polynomial functors over $\S$, define the $\i$-category $\P$ of such objects, and describe the mapping spaces in $\P$. We prove that every (small) limit and colimit of polynomials is still a polynomial in section 4. We prove that any composition of polynomials is still a polynomial in section 5. Section 6 is concerned with monoidal and comonoidal infinity structures in general, also the case of $\P$. Finally, we present our attempts toward the generalization of the Ahman-Uustalu theorem in the last section 7.

\subsection{Notations and conventions}
Throughout the paper, $\S$ denotes the $\i$-category of topological spaces, though we will often refer to the objects in $\S$ as $\i$-groupoids (\cite[remark 5.5.1.9]{Kerodon}). The straightening and unstraightening functors over a base $\i$-category $B$ are denoted by $St_B$ and $Un_B$, where the subscript might be suppressed when the base $B$ is clear. Limit and colimit are written as $\lim$ and $\colim$ respectively. We use $Map_{\mathcal{C}}(-,-)$ to denote the mapping spaces in an $\i$-category $\mathcal{C}$. We use $\Cat_\i$ to denote the ($\i$,1)-category of small ($\i$,1)-categories. For the reader's convenience, the comma category $\S_{/\S}$ that appears several times in the paper is introduced before \cref{sovers}.\par
Slightly abusing the language, we will often say ``$p:E\rightarrow B$ is a polynomial functor'' (in the bundle picture) or ``$p:B\rightarrow\S$ is a polynomial functor'' (in the diagram picture). Up to straightening-unstraightening equivalence, we will use the same notation $p$ indistinguishably. In both cases, the direction space at a position $b\in B$ is denoted as $p[b]$. For a morphism $\varphi$ of polynomial functors, we always use $\varphi_1$ and $\varphi^\sharp$ to denote the corresponding morphisms on positions and on directions.

\subsection{Acknowledgments}
The author would like to thank his advisor Guozhen Wang for sharing his insights on infinity category theory and for inspecting the draft of this paper. The author would like to thank David Spivak for valuable conversations and comments.

\section{Double limits and colimits}
In this section, we collect some preliminary results on the interchangeability of certain limits and colimits. These are mostly well known, but are hard to be located precisely in the literature. For the purpose of this paper, here we will only consider limits and colimits indexed by an $\i$-groupoid, in the category of $\i$-groupoids. A key ingredient is the \textit{straightening-unstraightening equivalence}:\begin{equation}
St_B:\S_{/B}\simeq Fun(B,\S):Un_B.
\end{equation}
More specifically,\begin{itemize}
\item The unstraightening is induced by the pullback of $\S_{*/}\rightarrow\S$ in $\Cat_\i$ (\cite[theorem 1.4]{clark}). Here $\S_{*/}$ is the undercategory on the final object $*\in\S$.
\item The straightening (at least on objects) is induced by taking the fibers: given $(E\xrightarrow{p}B)\in\S_{/B}$, as an $\S$-valued functor, $St_B(p)$ sends $b\in B$ to the (homotopy) pullback $\{b\}\times_BE$ in $\S$. The corresponding action on morphisms is constructed as \textit{covariant transport} in \cite[section 5.2.2]{Kerodon}.
\end{itemize}
The following theorem will be invoked repeatedly in this paper. It is really the central result that allows us to generalize polynomial functors from the setting of sets to that of infinity groupoids.
\begin{thm}
For $B\in\S$, given a functor $p:B\rightarrow\S$ viewed as a diagram indexed by $B$, let $E\rightarrow B$ be the unstraightening $Un_B(p)$ of $p$.\begin{itemize}
    \item The colimit of the diagram coincides with the total space of unstraightening up to homotopy equivalence (\cite[proposition 7.4.3.1]{Kerodon}):\begin{equation}\label{repeat1}
        \colim(p)\simeq E.
    \end{equation}
    \item The limit of the diagram coincides with the space of sections of unstraightening up to homotopy equivalence (\cite[proposition 7.4.1.6]{Kerodon}):\begin{equation}\label{repeat2}
        \lim(p)\simeq Map_{\S_{/B}}(B,E).
    \end{equation}
\end{itemize}
\end{thm}
\begin{cor}\label{constdiag}
    In the theorem above, taking $p$ to be constant on some $B^\prime\in\S$ implies that \begin{equation}
        \colim_BB^\prime\simeq B\times B^\prime,\ \ \ \lim_BB^\prime\simeq Map_{\S_{/B}}(B,B\times B^\prime)\simeq Map_{\S}(B,B^\prime).
    \end{equation}
\end{cor}
Let $\S_{/\S}\subset(\Cat_\i)_{/\S}$ denote the comma category of $\Cat_\i$ with respect to the inclusion functor $\iota:\S\rightarrow\Cat_\i$ and $\S\in\Cat_\i$. We think $\S_{/\S}$ as the category of \textit{diagrams of $\i$-groupoids indexed by $\i$-groupoid}, so taking colimit yields a functor $c:\S_{/\S}\rightarrow\S$. Since the forgetful functor $(\Cat_\i)_{/\S}\rightarrow\Cat_\i$ creates colimit (\cite[proposition 7.1.4.20]{Kerodon}), $(\Cat_\i)_{/\S}$ is cocomplete. Since $\iota$ preserves colimit, as a left adjoint (\cite[example 6.2.2.24]{Kerodon}), we deduce that the forgetting $\S_{/\S}\rightarrow\S$ again creates colimit so $\S_{/\S}$ is cocomplete.
\begin{prop}\label{sovers}
    Under the notations above, for any $B\in\S$ and any functor $f:B\rightarrow\S_{/\S}$, we have \begin{equation}
     \colim_B(c\circ f)\simeq c(\colim_B(f)).
    \end{equation}
\end{prop}
\begin{proof}
    For $b\in B$, say $f(b):B_b\rightarrow\S$ is a diagram. By the universal property of colimit, each $f(b)$ factors through a functor $q$ from $E:=\colim_{b\in B}B_b$ to $\S$. Since $\S_{/\S}\rightarrow\S$ creates colimit, this functor $q:E\rightarrow\S$ coincides with $\colim_B(f)$. Denote $Un(q):=(X\rightarrow E)$, so $X\simeq c(\colim_B(f))$ by \cref{repeat1}. Denote $Un(f(b)):=(X_b\rightarrow B_b)$, then $\colim_{b\in B}X_b\simeq\colim_B(c\circ f)$. Since $\S$ is locally Cartesian closed, each pullback in $\S$ preserves colimit, as a left adjoint. We thus deduce that $X_b\simeq B_b\times_EX$ by \cref{repeat1}. We then have a commutative diagram \begin{equation}
		\begin{tikzcd}
			X_b \arrow[r] \arrow[d] & B_b \arrow[d] \arrow[r] & \{b\} \arrow[d] \\
			X \arrow[r]             & E \arrow[r]                   & B              
		\end{tikzcd}
	\end{equation} in which both the left and right squares are Cartesian. By the pasting law, $X_b\simeq\{b\}\times_BX$. Again by \cref{repeat1}, we conclude that $X\simeq\colim_{b\in B}X_b$.
\end{proof}
\begin{rmk}
Under the notation $q:E\rightarrow\S$, we might rephrase this proposition in terms of double colimit:\begin{equation}\label{dbcolim1}
    \colim_E(q)\simeq\colim_B\colim_{B_b}(q|_{B_b}).
\end{equation}
\end{rmk}
\begin{prop}
    For $B,B^\prime\in\S$, given a functor $m:B\times B^\prime\rightarrow\S$, let $X\rightarrow B\times B^\prime$ be the unstraightening of $m$, we have\begin{equation}
        \varinjlim_B\varinjlim_{B^\prime}m(b,b^\prime)\simeq\varinjlim_{B^\prime}\varinjlim_Bm(b,b^\prime)\simeq X
    \end{equation}\begin{equation}\label{mapspace}
    \varprojlim_B\varinjlim_{B^\prime}m(b,b^\prime)\simeq\varinjlim_{f\in Fun(B,B^\prime)}\varprojlim_Bm(b,f(b))\simeq Map_{\mathcal{S}_{/B}}(B,X)
	\end{equation}where in the last equation, we view $X$ as $X\xrightarrow{Un(m)}B\times B^\prime\xrightarrow{proj_B}B$ in $\S_{/B}$.
\end{prop}
\begin{proof}
    For the first equation, it is known that colimits always commute with colimits. It remains to show that the double colimit coincides with $X$, a colimit over $B\times B^\prime$. But this is just \cref{dbcolim1} with $B_b=B^\prime$ constantly.\par
    For the last equation, it immediately follows from the first equation and \cref{repeat2} that $\varprojlim_B\varinjlim_{B^\prime}m(b,b^\prime)\simeq Map_{\mathcal{S}_{/B}}(B,X)$. It remains to show the other half. Note that $Fun(B,B^\prime)$ can be identified with the mapping space $Map_{\S_{/B}}(B,B\times B^\prime)$. We have two Cartesian diagrams \begin{equation}\label{bigdiag}
		\begin{tikzcd}
			{\varinjlim_Bm(b,f(b))} \arrow[d] \arrow[r] & X \arrow[d, "Un(m)"']                                                                &  & {\varprojlim_Bm(b,f(b))} \arrow[d] \arrow[r] & {Map_{\mathcal{S}_{/B}}(B,X)} \arrow[d]      \\
			B \arrow[r, "f"]                            & B\times B^\prime \arrow[rr, "{Map_{\mathcal{S}_{/B}}(B,-)}"', Rightarrow, shift left=9] &  & \{f\} \arrow[r]                              & {Map_{\mathcal{S}_{/B}}(B,B\times B^\prime)}
		\end{tikzcd}
	\end{equation} connected via the limit-preserving functor $Map_{\S_{/B}}(B,-)$. From the right Cartesian square and \cref{repeat1}, we conclude that $\varinjlim_f\varprojlim_Bm(b,f(b))\simeq Map_{\mathcal{S}_{/B}}(B,X)$.
\end{proof}
\begin{rmk}
    The last equation might be thought as analogue to the classical fact that $\Set$ is a \textit{distributive category}, i.e. products distribute over coproducts. 
\end{rmk}

\section{Polynomial functors}
\subsection{The main definition}
Because $\S$ is locally Cartesian closed, for any morphism $p:E\rightarrow B$ in $\S$, the pullback functor $p^*:\S_{/B}\rightarrow\S_{/E}$ admits a right adjoint, known as the \textit{dependent product} $p_*$. We state the first definition of polynomial functor over $\S$.
\begin{defn}\label{reviewdef}
    (\cite[definition 2.1.1]{gepner2022operads}) An endofunctor over $\S$ is called a polynomial functor, if for some morphism $p:E\rightarrow B$ in $\S$, up to natural equivalence it can be decomposed as \begin{equation}
        \S\xrightarrow{\times E}\S_{/E}\xrightarrow{p_*}\S_{/B}\xrightarrow{forget}\S.
    \end{equation}
\end{defn}
\begin{rmk}
    This is adapted verbatim from the definition of polynomial functor over $\Set$ in \cite[section 1.4]{gambino2013polynomial}. If the reader is wondering how such a definition captures the usual notion of polynomial, see \cite[section 1]{weber2015polynomials}. This should not be confused with the homonymic notion in Goodwillie calculus, which in general is not an endofunctor.
\end{rmk}

By the straightening-unstraightening equivalence, a morphism $E\rightarrow B$ in $\S$ can be equivalently viewed as a diagram from $B$ to $\S$. In this formalism, we have the following compatibility result.
\begin{thm}
    Let $B\xrightarrow{p}\S$, $b\mapsto p[b]$ be a diagram, denote $(E\xrightarrow{q}B):=Un(p)$. Then we have an equivalence in $\S_{/B}$, natural in $X\in\S$:\begin{equation}
        q_*(X\times E)\simeq\colim_BMap_{\S}(p[b],X),\ \ \ \forall X\in\S.
    \end{equation}
\end{thm}
\begin{proof}
    Under the straightening-unstraightening equivalence, we have two adjunction pairs, shown at once in the commutative diagram below:\begin{equation}
        \begin{tikzcd}
\S_{/B} \arrow[r, "\simeq"] \arrow[d, "q^*\dashv q_*"'] & {Fun(B,\S)} \arrow[d, "q_!\dashv Ran_q"] \\
\S_{/E} \arrow[r, "\simeq"]                             & {Fun(E,\S)}                             
\end{tikzcd}.
    \end{equation}
    Here $q_!$ is the precomposition with $q$, which by \cref{repeat1} and the fact that pullbacks in $\S$ preserve colimit, corresponds to $q^*$ in the commutative diagram. It then follows that the right Kan extension $Ran_q$ along $q$ completes the other commutative diagram for $q_*$. It suffices to prove the equivalence in $Fun(B,\S)$. Note that $St_E(X\times E)$ is a diagram constant on $X$, which we just denote by $X$ for brevity. By \cite[proposition 7.3.5.1]{Kerodon}, the Kan extension is computed as \begin{equation}
        Ran_q(X)(b)\simeq\lim_{b\rightarrow q(e)}X,\ \ \ \forall b\in B.
    \end{equation}Here the limit is taken over the comma category $b\downarrow q$ with respect to $\{b\}\hookrightarrow B\xleftarrow{q}E$. Note that every morphism $b\rightarrow q(e)$ admits Cartesian lift at $e\in E$, the universal property of Cartesian morphism implies that the subcategory $p[b]\subset b\downarrow q$ is coreflective, which then by the $\i$-categorical Quillen's theorem A implies that $p[b]\hookrightarrow b\downarrow q$ is left cofinal (see \cite[example 7.2.3.8]{Kerodon}), so the limit is reduced to the one over $p[b]$:\begin{equation}
        Ran_q(X)(b)\simeq\lim_{p[b]}X\simeq Map_{\S}(p[b],X),\ \ \ \forall b\in B.
    \end{equation}
    Here we have used \cref{constdiag}. We thus conclude by passing back to $\S_{/B}$.
\end{proof}
\begin{rmk}
    Let $y:\S^{op}\rightarrow Fun(\S,\S)$ denote the coYoneda embedding, then $y\circ p^{op}$ is a diagram of representable copresheaves indexed by $B^{op}$. However as an $\i$-groupoid, $B^{op}$ is canonically equivalent to $B$ so in the statement above we just write $\colim_B$. By writing $\colim_BMap_{\S}(p[b],X)$, we have implicitly used the fact that colimits and limits of presheaves are computed pointwise. The theorem above guarantees that the following second definition, which in this paper we will work with, is equivalent to the first one. 
\end{rmk}
\begin{defn}
    An endofunctor over $\S$ is called a polynomial functor, if for some $B\in\S$, up to natural equivalence it can be written as a colimit of representable copresheaves indexed by $B$. 
\end{defn}
\begin{rmk}
    If $p\in Fun(\S,\S)$ is a polynomial functor, $p$ determines the corresponding diagram $B\rightarrow\S$ up to equivalence. In fact, by \cref{constdiag}, $B\simeq p(*)$ is the evaluation of $p$ at the final object. We call $B$ the \textbf{position space} of $p$. The uniqueness of the functor $B\rightarrow\S$ is a consequence of Yoneda lemma, which will be clear after we discuss the mapping spaces in the next subsection.
\end{rmk}
We summarize the situation by pointing out that, when speaking of polynomial functors over $\S$, there are two distinct but equivalent pictures to bear in mind:\begin{itemize}
    \item The \textit{bundle picture}. The data of a polynomial functor is given by a morphism $E\xrightarrow{p}B$ in $\S$. Up to a fibrant replacement, we may always assume $p$ to be a Kan fibration, whence the terminology ``bundle picture''. We occasionally refer to $E$ and $B$ as ``total space'' and ``base space''. For $b\in B$, we call $p^{-1}(b):=\{b\}\times_BE$ the \textbf{direction space} of $p$ at $b$.
    \item The \textit{diagram picture}. The data of a polynomial functor is given by a functor $B\xrightarrow{p}\S$. The evaluation $p(b)$ is the direction space at $b\in B$. To avoid confusion with the case in the previous picture, we will always write $p[b]$ for the direction space.
\end{itemize}
The fact that we can go back and forth between these two pictures is rather trivial in $Poly$, when $E$ and $B$ are sets. In the case of $\i$-groupoids, the equivalence relies on the straightening-unstraightening equivalence as a non-trivial black box.

\subsection{Mapping spaces}
\begin{ntn}
    The full subcategory of the functor $\i$-category $Fun(\S,\S)$ spanned by all polynomial functors is denoted as $\P$.
\end{ntn}
We want to give a characterization of the mapping space between two polynomial functors, and describe the corresponding actions on positions and directions.
\begin{lem}
    For $B\in\S$ and $X,Y\in\S_{/B}$, denote the fibers at $b\in B$ by $X_b$ and $Y_b$ respectively, then \begin{equation}
     Map_{\S_{/B}}(X,Y)\simeq\lim_BMap_{\S}(X_b,Y_b).
    \end{equation}
\end{lem}
\begin{proof}
    This is a variant of \cref{repeat2}. By abuse of notation we view $X,Y$ as in $Fun(B,\S)$. Recall the definition of $\i$-categorical mapping space:\begin{equation}
        Map_{\S}(X_b,Y_b)\simeq X_b\times_{\S}Fun(\Delta^1,\S)\times_{\S}Y_b.
    \end{equation}Note that this is a limit over $\Lambda^2_2\coprod_{1\sim0}\Lambda^2_2$. On the other hand,\begin{equation}
    Map_{Fun(B,\S)}(X,Y)\simeq X\times_{Fun(B,\S)}Fun(B\times\Delta^1,\S)\times_{Fun(B,\S)}Y.
    \end{equation}We conclude by the interchangeability of a double limit over $B\times(\Lambda^2_2\coprod_{1\sim0}\Lambda^2_2)$.
\end{proof}

Let $p,p^\prime$ be two polynomial functors with position spaces $B$ and $B^\prime$ respectively. Let $m$ be the functor defined as\begin{equation}
    B^\prime\times B\simeq(B^{\prime})^{op}\times B\xrightarrow{(p^\prime)^{op}\times p}\S^{op}\times\S\xrightarrow{Map_{\S}}\S,\ (b^\prime,b)\mapsto Map_{\S}(p^\prime[b^\prime],p[b]).
\end{equation}
\begin{thm}
    Consider $X\xrightarrow{Un(m)}B^\prime\times B\xrightarrow{proj_B}B$, then the mapping space is equivalent to the space of sections of $X$ over $B$:\begin{equation}
        Map_{\P}(p,p^\prime)\simeq Map_{\S_{/B}}(B,X).
    \end{equation}
\end{thm}
\begin{proof}
    Since $p=\colim_{B}Map_{\S}(p[b],-)$, by the universal property of colimit:\begin{equation}
    Map_{\P}(p,p^\prime)\simeq\lim_BMap_{Fun(\S,\S)}(Map_{\S}(p[b],-),p^\prime).
    \end{equation}By the $\i$-categorical Yoneda lemma, we then have\begin{equation}
    Map_{\P}(p,p^\prime)\simeq\lim_Bp^\prime(p[b])\simeq\lim_B\colim_{B^\prime}Map_{\S}(p^\prime[b^\prime],p[b]).
    \end{equation}Now the result immediately follows from \cref{mapspace}.
\end{proof}
\begin{rmk}
    The proof above does not rely on the previous lemma. However, we want to pin down what is a point in $Map_{\S_{/B}}(B,X)$ in more detail. Recall in the proof of \cref{mapspace}, in the right diagram of \cref{bigdiag}, $Map_{\S_{/B}}(B,X)$ is equipped with a fibration into $Map_{\S}(B,B^\prime)$, whose fiber at a point $\varphi_1$ is $\lim_BMap_{\S}(p^\prime[\varphi_1(b)],p[b])$. For a morphism $\varphi:p\rightarrow p^\prime$, as a natural transformation, the evaluation at $*\in\S$ yields the corresponding morphism $\varphi_1:B\rightarrow B^\prime$ on positions. On directions, by the previous lemma, the fiber at $\varphi_1$ is exactly $Map_{\S_{/B}}(E^\prime\times_{B^\prime}B,E)$, in which a point $\varphi^\sharp$ can be viewed as a family of morphisms $\varphi^\sharp_b:p^\prime[\varphi_1(b)]\rightarrow p[b]$ indexed by $b\in B$. In summary, the data of a morphism $\varphi$ in $\P$ consists of $\varphi_1$ and $\varphi^\sharp$ making the diagram below commute \begin{equation}
        \begin{tikzcd}
			E \arrow[rd, "p"'] & E^\prime\times_{B^\prime}B \arrow[d] \arrow[r] \arrow[l, "\varphi^\sharp"'] & E^\prime \arrow[d, "p^\prime"] \\
			& B \arrow[r, "\varphi_1"]                                                    & B^\prime                      
		\end{tikzcd}.
    \end{equation}It is clear how morphisms compose: we have \begin{equation}\label{ssssss}
        (\psi\circ\varphi)_1\simeq\psi_1\circ\varphi_1,\ \ \ (\psi\circ\varphi)^\sharp\simeq\varphi_1^*(\psi^\sharp).
    \end{equation}Moreover, in a more concise way, the mapping space $Map_{\P}(p,p^\prime)$ should be identified with the solution space of the following horn filling problem in $\Cat_{\i}$ viewed as an $(\i,2)$-category: \begin{equation}\label{inftytwo}
		\begin{tikzcd}
			& B^\prime \arrow[rd, "p^\prime"] \arrow[d, "\varphi^\sharp", Rightarrow] &             &  & \Lambda_2^2 \arrow[d, hook] \arrow[r, "{(p,p^\prime)}"] & \Cat_\infty \\
			B \arrow[rr, "p"'] \arrow[ru, "\varphi_1"] & {}    & \mathcal{S} &  & \Delta^2 \arrow[ru, dashed]           &     
		\end{tikzcd}.
	\end{equation}The author could not make it precise though, due to the lack of expertise in $(\i,2)$-categories.
\end{rmk}
\begin{ex}\label{onlyexl}
    For any $B\in\S$, the constant functor $\S\rightarrow\S$ on $B$ is a polynomial functor whose corresponding diagram is $B\rightarrow\S$ constant on the initial object $\varnothing$. This embeds $\S$ as a full subcategory of $\P$. Moreover, the inclusion admits a left adjoint $\P\rightarrow\S$ which is the copresheaf induced by the identity functor $id_{\S}$: actually we have \begin{equation}
        Map_{\P}(id_{\S},p)\simeq B,\ \ \ \forall(B\xrightarrow{p}\S)\in\P.
    \end{equation}Similarly by unwinding the definition, the presheaf induced by $id_{\S}$ is the global section functor:\begin{equation}
        Map_{\P}(p,id_{\S})\simeq Map_{\S_{/B}}(B,E)\simeq\lim(p),\ \ \ \forall(E\xrightarrow{p}B)\in\P.
    \end{equation}
\end{ex}

\section{Completeness and cocompleteness}
\subsection{Limits}
In the case of $Poly$, it is shown in \cite[theorem 5.33]{niu2023polynomial} that any (small) limit of polynomials is again a polynomial, whose position space (resp. direction space) is the limit (resp. colimit) of the corresponding position spaces (resp. direction spaces). The following theorem is a counterpart in $\P$. 
\begin{thm}
    The $\i$-category $\P$ is complete.
\end{thm}
\begin{proof}
    Let $I$ be a simplicial set and $p_i:B_i\rightarrow\S$ be a family of polynomial functors indexed by $i\in I$. Take the limit $p:=\lim_Ip_i$ in $Fun(\S,\S)$, it suffices to prove that $p$ is a polynomial functor. If this were the case, since limits in $Fun(\S,\S)$ are computed pointwise, the position space of $p$ has to be \begin{equation}
        p(*)\simeq\lim_Ip_i(*)\simeq\lim_IB_i:=B.
    \end{equation} Let $q_i:B\rightarrow B_i$ ($i\in I$) be the structure morphisms in the limit cone. For any $b\in B$ and $X\in\S$, we have \begin{equation}
        Map_{\S}(p_i[q_i(b)],X)\simeq\{q_i(b)\}\times_{B_i}p_i(X),\ \ \ \forall i\in I.
    \end{equation}Note that $*\xrightarrow{q_i(b)}B_i\xleftarrow{p_i(X\rightarrow*)}p_i(X)$ can be viewed as a diagram $I\times\Lambda^2_2\rightarrow\S$, by the interchangeability of the double limit here, we have an equivalence \begin{equation}
        \{b\}\times_Bp(X)\simeq\lim_IMap_{\S}(p_i[q_i(b)],X)\simeq Map_{\S}(\colim_Ip_i[q_i(b)],X)
    \end{equation} natural in $b\in B$ and $X\in\S$. We then conclude by the straightening-unstraightening equivalence that \begin{equation}
        p\simeq\colim_BMap_{\S}(\colim_Ip_i[q_i(b)],-)
    \end{equation} hence $p$ is a polynomial functor.
\end{proof}
\begin{ex}
For two polynomial functors $p_i:B_i\rightarrow\S$ ($i=1,2$), their categorical product is given by \begin{equation}
    B_1\times B_2\xrightarrow{p_1\times p_2}\S\times\S\xrightarrow{\coprod}\S.
\end{equation} So on positions it takes the product, on directions it takes the coproduct.
\end{ex}
\begin{war}
    By the way we construct limit above, we see that the inclusion $\P\hookrightarrow Fun(\S,\S)$ is limit-preserving. In other words, limits in $\P$ are also computed pointwise, which can be seen as a consequence of Yoneda lemma:\begin{equation}
        (\varprojlim_Ip_i)(X)\simeq\varprojlim_Ip_i(X).
    \end{equation} However, this is no longer true for the colimits: in general \begin{equation}
        (\varinjlim_Ip_i)(X)\not\simeq\varinjlim_Ip_i(X)
    \end{equation} in $\P$, and the inclusion $\P\hookrightarrow Fun(\S,\S)$ does not preserve colimit (otherwise all endofunctors would be polynomials). This asymmetry can be thought as a reason why we treat the cases of limits and colimits so differently, and the latter case is much trickier, as we will soon see in the next subsection.
\end{war}

\subsection{Colimits}
For a colimit of polynomial functors, it is true that its position space is the colimit of position spaces (as we saw in \cref{onlyexl} that the functor $\P\rightarrow\S$ taking the position space is a left adjoint). However, we can no longer use the strategy of first taking the colimit in $Fun(\S,\S)$, nor can we describe the direction spaces very explicitly in general. We need to invoke \cite[remark 7.6.6.10]{Kerodon}, which reduces the goal of showing cocompleteness into showing the existences of small coproducts and coequalizers. The first part is relatively easy, and we have a stronger result below.
\begin{thm}\label{justlab}
    The colimit of polynomial functors indexed by an $\i$-groupoid is still a polynomial functor.
\end{thm}
\begin{proof}
    Let $B\in\S$ and $p_i:B_i\rightarrow\S$ be a family of polynomial functors indexed by $i\in B$. In this case, we can first take the colimit $\colim_Bp_i$ in $Fun(\S,\S)$, then prove that it is a polynomial. To this end, we denote $E:=\colim_BB_i$ and let $E\xrightarrow{p}\S$ be the colimit of $B_i\xrightarrow{p_i}\S$ in $\S_{/\S}$\footnote{It is generally not true, if $B$ is just a simplicial set, that a diagram $B\rightarrow\P$ would induce a diagram $B\rightarrow\S_{/\S}$, due to the $(\i,2)$-categorical issue indicated in \cref{inftytwo}. However, it is the case, when $B$ is an $(\infty,0)$-category.}. We then simply use \cref{dbcolim1}:\begin{equation}
    (\colim_Bp_i)(X)\simeq\colim_B\colim_{B_i}Map_{\S}(p_i[b_i],X)\simeq\colim_{e\in E}Map_{\S}(p(e),X),\ \forall X\in\S.
    \end{equation}We thus conclude that the colimit is the polynomial functor defined by the diagram $E\xrightarrow{p}\S$.
\end{proof}
\begin{cor}
    The $\i$-category $\P$ admits all small coproducts.
\end{cor}
\begin{proof}
    This is an immediate consequence of the theorem above, by viewing a set $J$ as a discrete space. In fact, in the bundle picture, the coproduct of $E_j\xrightarrow{p_j}B_j$ ($j\in J$) is just $\coprod_JE_j\xrightarrow{\coprod_Jp_j}\coprod_JB_j$.
\end{proof}

Next, for the hard part, the existence of coequalizers, we will need to invoke explicit point-set models of homotopy limit and homotopy colimit in $\S$ in the following arguments.
\begin{thm}
    For any pair of parallel morphisms $f,g:p_1\rightrightarrows p_2$ where $(E_i\xrightarrow{p_i}B_i)\in\P$ ($i=1,2$), their coequalizer exists.
\end{thm}
\begin{proof}
    Since $p_1$ is a colimit over $\i$-groupoid of representable copresheaves, by \cref{justlab} and the interchangeability of double colimit\footnote{The precise statement we use here is \cite[lemma 5.5.2.3]{HTT}: if one of the double colimits exists, then after interchanging, the other one also exists and coincides with the first one.}, to simplify the discussion, it suffices to deal with the case when $p_1$ itself is represented by some $X\in\S$. So we may assume that $B_1$ is trivial, and $p_1$ only has one direction space that is $X$. In this case, on positions $f_1$ and $g_1$ just pick a point $x\in B_2$ and $y\in B_2$ respectively. By the discussion at the beginning of this subsection, the position space of coequalizer must be the coequalizer of position spaces, which in this case is \begin{equation}
        B:=coeq(f_1,g_1:*\rightrightarrows B_2)\simeq B_2\coprod_{0\sim x,1\sim y}[0,1]
    \end{equation} by the explicit model of homotopy coequalizer. If the coequalizer $p:=coeq(f,g:p_1\rightrightarrows p_2)$ exists, then by Yoneda lemma $p$ is uniquely characterized by the property that $\forall(E^\prime\xrightarrow{p^\prime}B^\prime)\in\P$, there is an equivalence natural in $p^\prime$:\begin{equation}
    Map_{\P}(p,p^\prime)\simeq eq(f_!,g_!:Map_{\P}(p_2,p^\prime)\rightrightarrows Map_{\P}(p_1,p^\prime)).
    \end{equation}Denote $A_i:=Map_{\P}(p_i,p^\prime)$ ($i=1,2$) for brevity, we have the explicit model of homotopy equalizer\begin{equation}
        eq(f_!,g_!:A_2\rightrightarrows A_1)\simeq A_2\times_{A_1\times A_1}A_1^{[0,1]}.
    \end{equation}Unwinding the notions, a point in this space consists of the following data:\begin{itemize}
        \item A map $\psi:B\rightarrow B^\prime$.
        \item A morphism $\varphi\in A_2$ such that $\varphi_1:B_2\rightarrow B^\prime$ coincides with $\psi|_{B_2}$.
        \item A family of maps $h_t:p^\prime[\psi(t)]\rightarrow X$ indexed by $t\in[0,1]$, such that the two endpoints $h_0=f^\sharp\circ\varphi^\sharp_x$ and $h_1=g^\sharp\circ\varphi^\sharp_y$ (recall that $\psi(0)=\varphi_1(x)$ and $\psi(1)=\varphi_1(y)$ by the construction of $B$).
    \end{itemize} The goal is thus to find a polynomial functor $p$ with position space $B$, such that $Map_{\P}(p,p^\prime)$ is naturally equivalent to the space described above\footnote{Here it is tempting to take the total space $E$ of $p$ as gluing $E_2$ and $X\times[0,1]$ along the maps $f^\sharp:p_2[x]\rightarrow X$ and $g^\sharp:p_2[y]\rightarrow X$. However, the problem is that such $E\rightarrow B$ is not a fibration, and even after replacing it by the mapping path fibration, this still gives the wrong $p$. A modification of this idea leads to the construction that follows.}. Consider the simplicial set $B_2\coprod_{0\sim x,1\sim y}\Lambda_2^2$. We define a diagram $F$ over it in $\S$ by \begin{itemize}
        \item On $B_2$, $F|_{B_2}$ is given by $p_2:B_2\rightarrow\S$.
        \item On $\Lambda_2^2$, $F|_{\Lambda_2^2}$ is given by the diagram $p_2[x]\xrightarrow{f^\sharp}X\xleftarrow{g^\sharp}p_2[y]$.
    \end{itemize}Now we can take $p$ to be the right Kan extension of $F$ along the localization map $u:B_2\coprod_{0\sim x,1\sim y}\Lambda_2^2\rightarrow B$:\begin{equation}
        \begin{tikzcd}
      & {B_2\coprod_{0\sim x,1\sim y}[0,1]} \arrow[d, Rightarrow] \arrow[rd, bend right] \arrow[rd, "p", dashed, bend left] &  \\
{B_2\coprod_{0\sim x,1\sim y}\Lambda_2^2} \arrow[ru, "u"] \arrow[rr, "F"'] & {}                           & \mathcal{S}
\end{tikzcd}.
    \end{equation}We then conclude by the universal property of Kan extension that $p$ is the polynomial functor as desired.
\end{proof}
\begin{rmk}
    Our construction is inspired by \cite[theorem 5.43]{niu2023polynomial}, the case of $Poly$. As an example, if we denote by $y^X$ the polynomial represented by $X$ following the convention in $Poly$, then given two maps $f^\sharp:A\rightarrow X$ and $g^\sharp:B\rightarrow X$ of spaces, by the formula of Kan extension (\cite[proposition 7.3.5.1]{Kerodon}), the corresponding coequalizer of $y^X\rightrightarrows y^A+y^B$ is $y^{A\times_X^hB}$, where $A\times_X^hB$ is the homotopy pullback.
\end{rmk}
\begin{cor}
    The $\i$-category $\P$ is cocomplete.
\end{cor}

\section{Composition of polynomials}
\subsection{The composition product}
As endofunctors, polynomial functors can compose. This composition, in the classical case of $Poly$, can be easily written down as an explicit formula, which is just the incarnation of the usual algebraic rule of polynomial expansion: for every $p_1,p_2\in Poly$ we have (\cite[equation 21]{handbook})\begin{equation}
    p_1\circ p_2=\sum_{b_1\in B_1}(\sum_{b_2\in B_2}y^{p_2[b_2]})^{p_1[b_1]}=\sum_{b_1\in B_1,f:p_1[b_1]\rightarrow B_2}y^{\sum_{i\in p_1[b_1]}p_2[f(i)]}.
\end{equation} Hence a position of $p_1\circ p_2$ is a map $f:p_1[b_1]\rightarrow B_2$ for some position $b_1$ of $p_1$, and the direction set at $f$ is the pullback of $p_2$ along $f$. In \cite[theorem 2.1.8]{gepner2022operads}, the authors proved that compositions of polynomial functors over $\S$ are still polynomial functors, by giving an explicit description under their formalism. Since in this paper we have used another formalism, we will reprove the result here, also by giving an explicit description, which will make the analogy with the equation above transparent.
\begin{thm}
    For any $(E_i\xrightarrow{p_i}B_i)\in\P$ ($i=1,2$), $p_1\circ p_2\in Fun(\S,\S)$ is a polynomial functor.
\end{thm}
\begin{proof}
    If it were a polynomial functor, the evaluation at $*$ must produce its position space, which is \begin{equation}
        B:=p_1(p_2(*))\simeq p_1(B_2)\simeq\colim_{B_1}Map_{\S}(p_1[b_1],B_2).
    \end{equation} We have a family of diagrams in $\S$ indexed by $b_1\in B_1$, given by the following composition of functors:\begin{equation}
        Map_{\S}(p_1[b_1],B_2)\rightarrow\S_{/B_2}\xrightarrow{p_2^*}\S_{/E_2}\xrightarrow{forget}\S.
    \end{equation} Here the naturality in $b_1$ follows from the pasting law of pullbacks. Consider the resulting colimit in $\S_{/\S}$, denoted as $B\xrightarrow{p}\S$. We claim that $p_1\circ p_2$ is the polynomial functor $p$. The goal is to establish a natural equivalence $p_1(p_2(X))\simeq p(X)$, $\forall X\in\S$. First, by \cref{constdiag} we have \begin{equation}
        p_1(p_2(X))\simeq\colim_{B_1}Map_{\S}(p_1[b_1],p_2(X))\simeq\colim_{B_1}\lim_{p_1[b_1]}p_2(X).
    \end{equation}And by \cref{mapspace}, inside the colimit above we have \begin{align}
        \lim_{p_1[b_1]}p_2(X)&\simeq\lim_{p_1[b_1]}\colim_{B_2}Map_{\S}(p_2[b_2],X)\\&\simeq\colim_{f:p_1[b_1]\rightarrow B_2}\lim_{i\in p_1[b_1]}Map_{\S}(p_2[f(i)],X)\\&\simeq\colim_{f:p_1[b_1]\rightarrow B_2}Map_{\S}(\colim_{i\in p_1[b_1]}p_2[f(i)],X).
    \end{align}Furthermore, inside this colimit above, since pullbacks in $\S$ preserve colimit, we note that \begin{equation}
    \colim_{i\in p_1[b_1]}p_2[f(i)]\simeq\colim(p_1[b_1]\xrightarrow{f}B_2\xrightarrow{p_2}\S)\simeq f^*(E_2)
    \end{equation}which is exactly $p|_{Map_{\S}(p_1[b_1],B_2)}$ evaluated at $f$ by our definition of $p$. Combining the equations above finally yields\begin{equation}
        p_1(p_2(X))\simeq\colim_{B_1}\colim_{f:p_1[b_1]\rightarrow B_2}Map_{\S}(p(f),X)\simeq\colim_{b\in B}Map_{\S}(p[b],X)\simeq p(X).
    \end{equation}Here, we have used \cref{dbcolim1} to contract the double colimit.
    \end{proof}
    \begin{ex}\label{fuheji}
        For $(E\xrightarrow{p}B)\in\P$, let $p_n:E_n\rightarrow B_n$ denote the $n$-times composition $p^{\circ n}$ of $p$. The position spaces are given by a recurrence sequence:\begin{equation}
            B_{n+1}\simeq p(B_n)\simeq\colim_BMap_{\S}(p[b],B_n).
        \end{equation}The direction spaces are also determined recursively: once the total space $E_n$ is known, the direction space of $p_{n+1}$ at a position $f\in Map_{\S}(p[b],B_n)$ is $f^*(E_n)$.
    \end{ex}

\subsection{The coclosure}
The following lemma, in preparation for the next theorem, is well known in the case of ordinary categories.
\begin{lem}
    Let $F:\mathcal{C}\rightarrow\mathcal{D}$ be an $\i$-functor, then the left Kan extension along $F$ from $Fun(\mathcal{C},\S)$ to $Fun(\mathcal{D},\S)$ takes representables to representables:\begin{equation}
        Lan_F(Map_{\mathcal{C}}(c,-))\simeq Map_{\mathcal{D}}(F(c),-),\ \ \ \forall c\in\mathcal{C}.
    \end{equation}
\end{lem}
\begin{proof}
    Since $Lan_F$ is left adjoint to $F_!:Fun(\mathcal{D},\S)\rightarrow Fun(\mathcal{C},\S)$, there is an equivalence\begin{equation}
        Map_{Fun(\mathcal{D},\S)}(Lan_F(Map_{\mathcal{C}}(c,-)),G)\simeq Map_{Fun(\mathcal{C},\S)}(Map_{\mathcal{C}}(c,-),G\circ F)
    \end{equation}natural in $G\in Fun(\mathcal{D},\S)$. By Yoneda lemma, it follows that in $Fun(\mathcal{D},\S)$ we have a natural equivalence of mapping spaces \begin{equation}
    Map(Lan_F(Map_{\mathcal{C}}(c,-)),G)\simeq G(F(c))\simeq Map(Map_{\mathcal{D}}(F(c),-),G).
    \end{equation}We thus conclude that $Lan_F(Map_{\mathcal{C}}(c,-))\simeq Map_{\mathcal{D}}(F(c),-)$, again by Yoneda lemma.
\end{proof}
\begin{thm}
    For any $p\in\P$, the functor $-\circ p:\P\rightarrow\P$ induced by the precomposition with $p$ admits a left adjoint, denoted by ${p\brack-}$, defined by\begin{equation}
        {p\brack p_1}:=\varinjlim_{B_1}Map_{\mathcal{S}}(p(p_1[b_1]),-),\ \ \ \forall(B_1\xrightarrow{p_1}\S)\in\P.
    \end{equation}
\end{thm}
\begin{proof}
    Since on $Fun(\S,\S)$ we have an adjunction $Lan_p\dashv(-\circ p)$ and $\P\subset Fun(\S,\S)$ is a full subcategory, it suffices to show that $\P$ is closed under these two functors so that the adjunction descends to $\P$. We have proved the case of $-\circ p$ in the last subsection, it remains to prove that $Lan_p(p_1)$ is a polynomial, $\forall p_1\in\P$. But this immediately follows from the lemma above, \cref{justlab} and the fact that $Lan_p$ preserves colimit, as a left adjoint. Concretely, since $p_1$ is the colimit of the representables $Map_{\S}(p_1[b_1],-)$, we derive the formula \begin{equation}
        Lan_p(p_1)\simeq\colim_{B_1}Lan_p(Map_{\S}(p_1[b_1],-))\simeq\colim_{B_1}Map_{\S}(p(p_1[b_1]),-)={p\brack p_1}
    \end{equation} by the lemma above. 
\end{proof}
\begin{rmk}
    This result is a straightforward generalization of the coclosure formula for $\circ$ in $Poly$ (\cite[equation 68]{handbook}). Alternatively, we can characterize this adjunction by specifying its unit. By the universal property of colimit and Yoneda lemma, there is an equivalence natural in $p_1,p_2\in\P$:\begin{align}
		Map_{Poly_{\mathcal{S}}}({p\brack p_1},p_2)&\simeq\varprojlim_{B_1}Map_{Poly_{\mathcal{S}}}(Map_{\mathcal{S}}(p(p_1[b_1]),-),p_2)\\&\simeq\varprojlim_{B_1}(p_2\circ p)(p_1[b_1])\\&\simeq\varprojlim_{B_1}Map_{Poly_{\mathcal{S}}}(Map_{\mathcal{S}}(p_1[b_1],-),p_2\circ p)\\&\simeq Map_{Poly_{\mathcal{S}}}(p_1,p_2\circ p).
	\end{align} This equivalence is induced by the unit transformation $p_1\rightarrow{p\brack p_1}\circ p$, defined by the obvious map \begin{equation}
	    \colim_{B_1}Map_{\S}(p_1[b_1],-)\xrightarrow{p}\colim_{B_1}Map_{\S}(p(p_1[b_1]),p(-)).
	\end{equation}Besides, there are two more monoidal products in $Poly$, the categorical product and the so-called \textit{Dirichlet product}, that both admit closure with explicit construction using the composition product (see \cite[section 4]{handbook}). We expect no obstacles in generalizing these structures to the setting of $\P$. However, we will not pursue this line in the current paper.
\end{rmk}

\section{Monoidal infinity categories and comonads}
\subsection{Monoidal infinity structure revisited}
Following the standard definition coined by Lurie, one typically understands a monoidal $\i$-category as a coCartesian fibration $\mathcal{C}^\otimes\rightarrow\Delta^{op}$\footnote{Here $\Delta$ stands for the simplex category, viewed as an $\i$-category via taking its nerve.} satisfying the Segal condition (\cite[definition 1.1.2]{feijiaohuanl}), and one often says informally that $\mathcal{C}:=\mathcal{C}^\otimes_{[1]}$ which denotes the fiber at $[1]\in\Delta^{op}$ is a monoidal $\i$-category. Under this formalism, the monoidal product in $\mathcal{C}$ is encoded in the coCartesian lifts of $[1]\cong\{0,2\}\hookrightarrow\{0,1,2\}=[2]$. The notion of monoid objects in $\mathcal{C}$ is defined in \cite[definition 1.1.14]{feijiaohuanl}, which roughly mimics the classical fact that a monoid object in an ordinary monoidal category $C$ can be identified with a lax monoidal functor from the trivial category $\mathbf{1}$ to $C$.\par
However, in this paper we are concerned with comonoid objects, and this notion is not so easily encoded in the aforementioned formalism. As in the ordinary categorical case, one may wish to define a comonoid object in $\mathcal{C}$ as an oplax monoidal functor from the trivial category $\mathbf{1}$ to $\mathcal{C}$, and in addition define an oplax monoidal functor as a lax monoidal functor between opposite monoidal categories. This is not wrong technically, but as pointed out by Lurie in \cite[remark 1.1.11]{feijiaohuanl}, in his formalism the notion of oplax monoidal (which he called left monoidal) functors has no easy definition. We explain the key issue here: when $\mathcal{C}^\otimes\rightarrow\Delta^{op}$ represents a monoidal infinity structure, the fact that the opposite category $\mathcal{C}^{op}$ also inherits a monoidal structure, which is rather trivial in the ordinary categorical case, becomes very tricky now. The naive attempt $(\mathcal{C}^{op})^\otimes\simeq(\mathcal{C}^\otimes)^{op}$ is not expected to be correct: after all, there is no evident choice of coCartesian fibration $(\mathcal{C}^\otimes)^{op}\rightarrow\Delta^{op}$. In fact, the correct description of opposite monoidal structure involves the notion of \textit{dual fibration} worked out in detail in \cite{barwick2018dualizing}. In our paper, we will not and do not have to dive into the explicit construction of dual fibration. Instead, we will sidestep the subtle issue discussed here by tactically working with an alternative formalism, which can be viewed as an ``opposite version'' of Lurie's definition:
\begin{defn}
    \footnote{The author was first informed of this definition in \cite[remark 3.3]{torii2025duoidal}.}A Cartesian fibration $\mathcal{C}^\otimes\rightarrow\Delta$ is said to be a monoidal infinity structure, if it satisfies the Segal condition: denoting by $\mathcal{C}^\otimes_{[n]}$ the fiber at $[n]\in\Delta$, there is an equivalence of $\i$-categories \begin{equation}
        \mathcal{C}^\otimes_{[n]}\xrightarrow{\sim}\prod_{i=1}^n\mathcal{C}^\otimes_{\{i-1,i\}}\simeq(\mathcal{C}^\otimes_{[1]})^n
    \end{equation} induced by the inclusion of spine in the standard $n$-simplex, for each $n\in\mathbb{N}$. We often say that $\mathcal{C}:=\mathcal{C}^\otimes_{[1]}$ is a monoidal $\i$-category for simplicity.
\end{defn}
It is clear that the two definitions are equivalent: in fact, under the straightening-unstraightening equivalence, in both formalisms the data amounts to a ``simplicial $\i$-category'' $\Delta^{op}\rightarrow\Cat_{\i}$ that sends $[n]$ to $\mathcal{C}^n$ for all $n$. More specifically, the relation between the two formalisms is clarified as follows:\begin{itemize}
    \item If the monoidal structure of $\mathcal{C}$ is given by coCartesian fibration $\mathcal{C}^\otimes\rightarrow\Delta^{op}$, then passing to the dual, it is also given by the Cartesian fibration $(\mathcal{C}^\otimes)^\vee\rightarrow\Delta$ (see the table in \cite[section 1.1]{barwick2018dualizing}), and vice versa.
    \item On the other hand, passing to the opposite, if a fibration (coCartesian or Cartesian) satisfies one of the definitions, then its opposite fibration satisfies the other one and vice versa. That is to say, if the monoidal structure of $\mathcal{C}$ is given by one side, then the monoidal structure of $\mathcal{C}^{op}$ is better given by the other side. If sticking to one formalism, for example the original one, then the monoidal structure $\mathcal{C}^\otimes\rightarrow\Delta^{op}$ on $\mathcal{C}$ induces an opposite monoidal structure $((\mathcal{C}^\otimes)^\vee)^{op}\simeq(\mathcal{C}^{op})^\otimes\rightarrow\Delta^{op}$ on $\mathcal{C}^{op}$.
\end{itemize}
Similarly, the original formalism is able to capture the notion of lax monoidal functor but not that of oplax monoidal functor, while in our formalism, we can easily define oplax monoidal functor, at the expense of the notion of lax monoidal functor, which will not be relevant to the topic of our paper anyway.
\begin{defn}\label{comoddef}
    Let $\mathcal{C}^\otimes\rightarrow\Delta$ be a monoidal $\i$-category. A comonoid object in $\mathcal{C}$ is an oplax monoidal functor from $\mathbf{1}$ to $\mathcal{C}$, that is, a section $q:\Delta\rightarrow\mathcal{C}^\otimes$ such that every \textit{convex morphism} $\rho^i_n:[1]\cong\{i-1,i\}\hookrightarrow[n]$ ($1\leq i\leq n$) is sent by $q$ to a Cartesian morphism. We informally say that $q([1])$ is the comonoid object. 
\end{defn}

\subsection{Polynomial comonads}
Now we specialize the discussion made in the last subsection to the case of $\P$. As a full subcategory of $Fun(\S,\S)$ which is closed under the composition product, $\P$ inherits the monoidal infinity structure of $Fun(\S,\S)$ induced by $\circ$ (\cite[proposition 1.3.1.1]{feijiaohuanl}), with $id_{\S}$ being the monoidal unit. The fact that $Fun(\S,\S)$ is indeed a monoidal $\i$-category (\cite[proposition 3.1.2.1]{feijiaohuanl}), which seems very plausible at least in the ordinary categorical case, takes some efforts to work out: first, viewing $Fun(\S,\S)$ as a simplicial monoid under the composition, we can apply the classifying space functor $B$ to it (\cite[construction 1.3.2.5]{Kerodon}), to obtain a functor $F:\Delta^{op}\rightarrow\Cat_{\i}$ that sends $[n]$ to $Fun(\S,\S)^n$, $\forall n\in\mathbb{N}$. This is the resulting monoidal structure before unstraightening. Then Lurie takes the weighted nerve of $F$ to obtain a coCartesian fibration over $\Delta^{op}$ (\cite[corollary 5.3.3.16]{Kerodon}), while in our formalism we need to consider the Cartesian fibration that $F$ classifies. The upshot in the case of $\P$ is that, there is a Cartesian fibration $\P^\otimes\rightarrow\Delta$ such that \begin{itemize}
    \item The fiber at $[n]\in\Delta$ can be identified with $\P^n$, whose objects are $n$-tuples of polynomial functors.
    \item Over a morphism $f:[n]\rightarrow[m]$ in $\Delta$, a morphism $(p_1,\cdots,p_n)\rightarrow(p_1^\prime,\cdots,p_m^\prime)$ in $\P^\otimes$ consists of the following data: for each $1\leq i\leq n$, a morphism \begin{equation}
    \varphi^i:p_i\rightarrow p_{f(i-1)+1}^\prime\circ\cdots\circ p^\prime_{f(i)}    
    \end{equation}
    in $\P$. This morphism is Cartesian precisely when each $\varphi^i$ is an equivalence in $\P$. For example, the composition product in $\P$ is encoded in the Cartesian lifts of $[1]\cong\{0,2\}\hookrightarrow\{0,1,2\}=[2]$, and the monoidal unit $id_{\S}$ is specified by the Cartesian lift of $[1]\twoheadrightarrow[0]$. 
\end{itemize} 
To unravel \cref{comoddef} in the case of $\P$ in more detail, we first review what are the polynomial comonads in $Poly$.
\begin{defn}
    A polynomial comonad over $\Set$, is an object $p\in Poly$ equipped with a \textit{counit} $\epsilon$ from $p$ to $y:=id_{\Set}$, and a \textit{comultiplication} $\delta:p\rightarrow p\circ p$, satisfying the following compatibility conditions:\begin{itemize}
        \item The unit law: the diagram below commutes in $Poly$.\begin{equation}
			\begin{tikzcd}
				y\circ p \arrow[r, no head, Leftarrow] & p \arrow[r, no head, Rightarrow] \arrow[d, "\delta"']     & p\circ y \\
				& p\circ p \arrow[ru, "id\circ\epsilon"'] \arrow[lu, "\epsilon\circ id"] &         
			\end{tikzcd}.
		\end{equation}
        \item The associative law: the diagram below commutes in $Poly$.\begin{equation}
			\begin{tikzcd}
				p \arrow[r, "\delta"] \arrow[d, "\delta"'] & p\circ p \arrow[d, "id\circ\delta"] \\
				p\circ p \arrow[r, "\delta\circ id"']      & p\circ p\circ p 
			\end{tikzcd}.
		\end{equation}
    \end{itemize}
\end{defn}
\begin{ntn}\label{shouldi}
    We adopt the standard notation that in $\Delta$, the \textit{face map} $d^n_i:[n-1]\rightarrow[n]$ is the one whose values only omit $i$ ($0\leq i\leq n$); the \textit{degeneracy map} $s^n_i:[n+1]\rightarrow[n]$ is the one whose values only duplicate $i$ ($0\leq i\leq n$). We often suppress the superscript for brevity when there is no risk of confusion.    
\end{ntn}
Over $\S$, the comonoid structure is encoded in a section $q:\Delta\rightarrow\P^\otimes$, whose image $p:=q([1])$ is the underlying polynomial comonad in $\P$. The condition that $q(\rho^i_n)$ is a Cartesian lift of $\rho^i_n$ says nothing but that the image $q([n])\in\P^n$ is precisely $(p,p,\cdots,p)$, for all $n$. To compare with the classical definition above, we note that the counit $\epsilon:p\rightarrow id_{\S}$ and the comultiplication $\delta:p\rightarrow p\circ p$ are induced by $q([1]\twoheadrightarrow[0])$ and $q([1]\xhookrightarrow{non\ convex}[2])$ respectively. Moreover, the unit law and the associative law boil down to the commutativity of \begin{equation}
    \begin{tikzcd}
{[1]} \arrow[r, "d_1", hook] \arrow[d, "d_1"', hook] & {[2]} \arrow[d, "s_0", two heads] &  & {[1]} \arrow[r, "d_1", hook] \arrow[d, "d_1"', hook] & {[2]} \arrow[d, "d_1", hook] \\
{[2]} \arrow[r, "s_1"', two heads]                   & {[1]}                             &  & {[2]} \arrow[r, "d_2"', hook]                        & {[3]}   
\end{tikzcd}
\end{equation} in $\Delta$. In the classical definition, there are only finitely many compatibility conditions to check, terminating at $p\circ p\circ p$. But in the $\i$-categorical case, there are infinitely many higher homotopy coherence structures involving all $p^{\circ n}$, which roughly speaking, are stored in the combinatorics in $\Delta$.

\section{A construction on polynomial comonads}
\subsection{The theorem of Ahman-Uustalu}\label{lablesubsec}
Before discussing a possible generalization of the Ahman-Uustalu theorem for $\P$, we first review its classical context. Starting with a polynomial comonad $p:E\rightarrow B$ over $\Set$, there is a canonical construction of turning $p$ into a category $C$, following the dictionary described below:\begin{itemize}
    \item Objects = positions: $Ob(C)=B$; arrows = directions: $Mor(C)=E$.
    \item The map $p:E\rightarrow B$ specifies the source object of an arrow.
    \item The comultiplication $\delta:p\rightarrow p\circ p$ on positions picks a map $f_b:p[b]\rightarrow B$ for each $b\in B$. Together $\sum_{b\in B}f_b$ induces a map $q:E\rightarrow B$ that specifies the target object.
    \item For each $b\in B$, the comultiplication on directions $\delta^\sharp_b$ determines an element $e\in p[b]$ for each $e_1\in p[b]$ and $e_2\in p[f_b(e_1)]$, which in the category $C$ is exactly the composition $e_2\circ e_1$. The associative law of arrow compositions is encoded in the associative law of $\delta$.
    \item The counit $\epsilon:p\rightarrow y$ as a global section of $p$, for each $b\in B$ picks an element $id_b\in p[b]$, which is the identity morphism at $b$. Identity morphisms satisfy various compatibility conditions encoded in the unit law.
\end{itemize} In summary, via the above procedure the whole structure of a small category can be repackaged and stored in the comonoid structure of a polynomial as equivalent data. The precise statement is as follows.
\begin{thm}
    There is an isomorphism (not just an equivalence) of categories, between the category of comonoid objects in $Poly$, and the category of small categories with retrofunctors (see \cite[definition 3.21]{spivak2025functorial}) among them. 
\end{thm}
We refer to \cite[chapter 7]{niu2023polynomial} for an elaborated exposition on this theorem. Here, we would like to give an informal account on this correspondence, in order to gain some guideline for the later construction. As shown in the figure below, in the first column, a position in $p$ is depicted as a dot, thought as an object in the category $C$. And pretty intuitively, a direction at this dot is an arrow emanating from it, but the target dot is temporarily indecisive: \begin{figure}[H]
\centering
\includegraphics[scale=0.2]{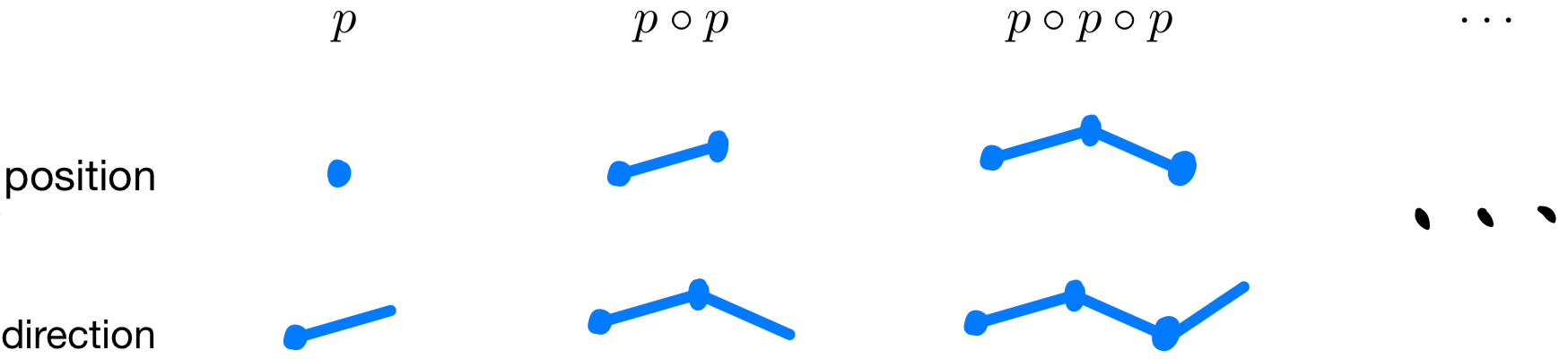}
\caption{Self compositions of a polynomial}
\end{figure}
As we move on to the second column, the positions of $p\circ p$ are thought as ``all possible ways of associating destinations to the arrows'', and the map $\delta_1$ picks the correct one in $C$. So strictly speaking, what we depict in the upper row of $p\circ p$ really means ``those positions coming from $B$''. To form a direction at a such position, we attach a new arrow to it, whose target is again missing, as for now. How we contract these two consecutive arrows is recorded in $\delta^\sharp$. Next, the missing target in the second column will be decided in those positions of $p^{\circ3}$ coming from $B$, and we expand another hanging arrow, so on and so forth. Eventually, the whole category $C$ grows out from such process.\par
As an $\i$-categorical analogue, following the discussion made in our introduction section, we may formulate a conjecture as follows.
\begin{conj}\label{mainconj}
    Given a polynomial comonad $p$ over $\S$, there is a canonical construction on $p$ that outputs a complete Segal space. Moreover, up to equivalence every complete Segal space arises from a such construction.
\end{conj}
In the rest of the paper, we will be concerned with one direction that goes from a polynomial comonad to a Segal space. We have not fully succeeded, the best we can offer is the following result, that will be proved in \cref{lastsubsec}. 
\begin{thm}\label{zby}
    If $(E\xrightarrow{p}B)\in\P$ is a comonoid object, then there exists an augmented cosimplicial space $X:\Delta_+\rightarrow\S$, such that \begin{equation}
        X_{-1}\simeq B,\ \ X_0\simeq E,\ \ X_1\simeq E\times_BE,\ \ \cdots,\ \ X_n\simeq\underbrace{E\times_BE\cdots \times_BE}_{(n+1)\textrm{ copies of }E},\ \ \cdots.
    \end{equation}Here the limit is taken over $E\xrightarrow{q}B\xleftarrow{p}E\xrightarrow{q}B\cdots\xleftarrow{p}E$, in which the morphism $q$ is induced by $\delta_1$, the comultiplication on positions.
\end{thm}
\begin{rmk}
    Even we manage to prove \cref{mainconj}, it is not clear how to promote it into a categorical equivalence as in the ordinary categorical case. After all, we have not studied $\i$-categorical retrofunctors yet. We note that, via the coYoneda embedding, the opposite category of representable polynomial comonads is equivalent to the category of $A_\i$-spaces. The latter category is known to be equivalent to the $\i$-category of small pointed connected $\i$-categories via delooping. So our conjectural equivalence should subsume this result and in some sense serve as a many-object generalization of it, which seems to be daunting.
\end{rmk}

\subsection{A d{\'e}calage-like functor}
We construct a functor denoted as $e$, which will play an important role in the proof of \cref{zby}. We arrange this construction as a separated subsection, because the author speculates that it might be of independent interests.  
\begin{cons}\label{maincons}
    Consider $e:\Delta^{op}\rightarrow\Delta$ defined by \begin{itemize}
    \item On objects, $e([n-1])=[n]$, $\forall n\in\mathbb{N}^*$.
    \item On morphisms, for a morphism $[m-1]\rightarrow[n-1]$ in $\Delta^{op}$ expressed as a non-strictly increasing map $f:\{1,\cdots,n\}\rightarrow\{1,\cdots,m\}$, the image $e(f):[m]\rightarrow[n]$ is the map that for each $i=0,1,\cdots,m$, sends $i$ to the maximal element in $f^{-1}(\leq i)$\footnote{If $f^{-1}(\leq i)=\varnothing$, then $i\mapsto0$. Equivalently to incorporate this special case, we may add a supplementary value $f(0)=0$ to the map $f$.}. 
    \end{itemize}
\end{cons}
\begin{prop}\label{zuow}
    In the construction above, $e$ is a faithful functor whose image $\tilde{\Delta}$ is the subcategory of $\Delta$ defined by \begin{equation}
    Ob(\tilde{\Delta})=\{[n]|n\geq1\},\ Mor(\tilde{\Delta})=\{f:[m]\rightarrow[n]|f(0)=0,f(m)=n\}.
    \end{equation}
\end{prop}
\begin{proof}
    For the functoriality, given two consecutive non-strictly increasing maps\begin{equation}
        \{1,\cdots,n\}\xrightarrow{f}\{1,\cdots,m\}\xrightarrow{g}\{1,\cdots,r\}
    \end{equation}we need to verify that $e(g\circ f)=e(f)\circ e(g):[r]\rightarrow[n]$. For each $i\in[r]$, assume $i\overset{e(g)}{\mapsto}j\overset{e(f)}{\mapsto}k$. From the definition of $e$ we deduce that \begin{equation}
        g(j)\leq i<i+1\leq g(j+1),\ \ \ f(k)\leq j<j+1\leq f(k+1).
    \end{equation} Since $g$ is order-preserving, these two inequalities above imply that \begin{equation}
        g(f(k))\leq g(j)\leq i<i+1\leq g(j+1)\leq g(f(k+1)).
    \end{equation} Therefore $k=e(g\circ f)(i)$. Clearly $e$ also preserves identity morphisms, so $e$ is a functor. Note that $e(f)$ always preserves the minimal and the maximal elements. Conversely, given a such map $\tilde{f}:[m]\rightarrow[n]$, it uniquely determines a map $f:[m-1]\rightarrow[n-1]$ in $\Delta^{op}$ such that $e(f)=\tilde{f}$ as follows:\begin{equation}
        \tilde{f}(i-1)+1,\cdots,\tilde{f}(i)\overset{f}{\mapsto}i,\ \ \ \forall 1\leq i\leq m.
    \end{equation} Hence $e$ is faithful and has the image $\tilde{\Delta}$ as desired.\par
    In fact, for the functoriality we have an alternative proof based on describing \cref{maincons} in terms of face and degeneracy maps (which might be a better way to understand the functor $e$). Note that $e$ sends $d_i^n$ to $s_i^n$ ($\forall0\leq i\leq n$), and $s_{i-1}^{n-1}$ to $d_i^{n+1}$ ($\forall1\leq i\leq n$). The relations among these generators are known as the simplicial identities: \begin{align}
        d_j^nd_i^{n-1}&=d_i^nd_{j-1}^{n-1},\ \ \ (i<j).\\
        s_j^nd_i^{n+1}&=d_i^ns_{j-1}^{n-1},\ \ \ (i<j).\\
        s_j^nd_i^{n+1}&=id,\ \ \ \ \ \ \ \ \ \ \ (i=j,j+1).\\
        s_j^nd_i^{n+1}&=d_{i-1}^ns_j^{n-1},\ \ \ (i>j+1).\\
        s_j^{n-1}s_i^n&=s_i^{n-1}s_{j+1}^n,\ \ \ (i\leq j).
    \end{align}
    It is then tedious but straightforward to verify that the $i$-th equation above is sent by $e$ to the $(6-i)$-th equation, $\forall 1\leq i\leq5$ (note the contravariance of $e$). Hence the relations are preserved, $e$ is a functor.
\end{proof}

Given a simplicial object $X$ (in any category), slightly abusing \cref{shouldi}, we still use $d_i$ and $s_i$ to denote the face and degeneracy morphisms in $X$:\begin{equation}
	X_0\overset{s_0}{\underset{d_0,d_1}\rightleftarrows}X_1\overset{s_0,s_1}{\underset{d_0,d_1,d_2}\rightleftarrows}X_2\overset{s_0,s_1,s_2}{\underset{d_0,d_1,d_2,d_3}\rightleftarrows}X_3\rightleftarrows\cdots.
\end{equation}By precomposing with $e^{op}$, we obtain a cosimplicial object \begin{equation}
	X_1\overset{s_0,s_1}{\underset{d_1}\rightleftarrows}X_2\overset{s_0,s_1,s_2}{\underset{d_1,d_2}\rightleftarrows}X_3\rightleftarrows\cdots.
\end{equation}We summarize by saying that inside every simplicial object sits a cosimplicial object, and conversely inside every cosimplicial object sits a simplicial object.
\begin{rmk}
    This shifting effect described above should remind one of the notion of \textit{d{\'e}calage} introduced in \cite{illusie2006complexe}. While our functor $e$ discards all the outer face morphisms $d_0^n,d_n^n$ and serves as a transition between simplicial and cosimplicial objects, by discarding all the last face and degeneracy morphisms $s_n^n,d_n^n$, the d{\'e}calage of a simplicial object produces a new simplicial object, and is related to path space objects. To the best of the author's knowledge, our shifting functor has not been made explicit in published literature, but there is a deja vu in \cite[construction 4.1.2.9]{HA}: what Lurie called a cut functor is in some sense an inverse to our construction. Anyway, the functor $e$ might be an interesting topic to study in its own right, which we will leave for future investigation.  
\end{rmk}
\begin{var}\label{finsab}
    Let $\Delta_+$ denote the augmented simplex category. We can extend $e^{op}$ to a functor $e_+:\Delta_+\rightarrow\Delta^{op}$, by requiring that $e_+$ sends the initial object $[-1]:=\varnothing\in\Delta_+$ to the initial object $[0]\in\Delta^{op}$. To strengthen our summary, we can now say that inside every simplicial object sits an augmented cosimplicial object. 
\end{var}

\subsection{The main construction}\label{lastsubsec}
We return to \cref{mainconj}. Naively, given a polynomial comonad $E\xrightarrow{p}B$ over $\S$, we may expect that there is a simplicial space $\Delta^{op}\rightarrow\S$ that sends $[n]$ to $E\times_BE\cdots \times_BE$ (there are $n$ copies of $E$). In fact, it is not hard to write down all the face and degeneracy morphisms here. But as pointed out in \cite[warning 10.2.0.5]{Kerodon}, unlike the 1-categorical case, to determine a simplicial object it does not suffice to just specify the morphism generators, as is always when working with $\i$-categories. We must find a slick way to argue. Another problem is that, even we can construct such a Segal space, there is no reason for it to be complete. We might need to apply the completion functor to it (\cite[section 14]{rezk2001model}), which seems to be unsatisfactory.\par
We will content ourselves with a partial result here: roughly speaking, we will construct this conjectural Segal space after precomposing with the functor $e_+$ in \cref{finsab}. Inspired by the graphic illustration we made in \cref{lablesubsec}, and adopting the notations in \cref{fuheji} (so in particular $E_0=B_0=*$, $E_1=E$ and $B_1=B$), we will consider the spaces \begin{equation}
    X_{n-1}:=B\times_{B_n}E_n,\ \ \ n\in\mathbb{N}.
\end{equation}Here the morphism $B\rightarrow B_n$ is the positional morphism of $p\rightarrow p^{\circ n}$, induced by $[1]\cong\{0,n\}\hookrightarrow[n]$. In what follows, we will prove \cref{zby}, by showing that the spaces $X_{-1},X_0,\cdots$ form the augmented cosimplicial space as desired.
\begin{proof}
    We divide the proof into two parts.\begin{enumerate}
    \item In the first part, we will construct the augmented cosimplicial structure. Let $s:\Delta\rightarrow\P^{\otimes}$ be the section that defines the comonoidal structure on $p$. Consider the image of the faithful functor $e_+^{op}:\Delta_+^{op}\rightarrow\Delta$, denoted as $\tilde{\Delta}_+$. Then $\tilde{\Delta}_+$ is just $\tilde{\Delta}$ from \cref{zuow} with the final object $[0]$ adding to it. Consider the Cartesian fibration $\widetilde{\P^{\otimes}}\rightarrow\tilde{\Delta}_+$ restricted from the monoidal structure $\P^{\otimes}\rightarrow\Delta$. Since $\tilde{\Delta}_+$ has an initial object $[1]$ and $\P$ is exactly the fiber at it, we have a functor $f:\widetilde{\P^{\otimes}}\rightarrow\P$. Specifically, on the fiber at $[n]$ this functor $f$ is given by the Cartesian lift of $[1]\cong\{0,n\}\hookrightarrow[n]$ in $\tilde{\Delta}_+$. Precomposing $f$ with the restricted section $s$ on $\tilde{\Delta}_+$ yields a functor \begin{equation}
        \tilde{X}:\Delta_+^{op}\cong\tilde{\Delta}_+\xrightarrow{s}\widetilde{\P^{\otimes}}\xrightarrow{f}\P.
    \end{equation} The upshot is that we obtain an augmented simplicial object $\tilde{X}$ in $\P$: \begin{equation}\label{augsump}
			id_{\mathcal{S}}\underset{d_0}\leftarrow p\overset{s_0}{\underset{d_0,d_1}\rightleftarrows}p\circ p\overset{s_0,s_1}{\underset{d_0,d_1,d_2}\rightleftarrows}p\circ p\circ p\overset{s_0,s_1,s_2}{\underset{d_0,d_1,d_2,d_3}\rightleftarrows}p^{\circ4}\rightleftarrows\cdots.
		\end{equation}Here, all the face morphisms are induced by the counit $\epsilon$ and all the degeneracy morphisms are induced by the comultiplication $\delta$, in particular the first $d_0$ and $s_0$ are exactly $\epsilon$ and $\delta$. Next, consider given a morphism $\varphi:p\rightarrow p^\prime$ in $\P$, the induced morphism on directions $\varphi^\sharp:E^\prime\times_{B^\prime}B\rightarrow E$ can be viewed as an object in $\S_{/E}$. This defines a functor \begin{equation}
		    g:(\P)^{op}_{p/}\rightarrow\S_{/E},\ \ \ \varphi\mapsto\varphi^\sharp.
		\end{equation} We can check the functoriality here on mapping spaces, which is essentially \cref{ssssss} and the pasting law:\begin{equation}
		Map_{p/}(p\xrightarrow{\varphi}p^\prime,p\xrightarrow{\psi}p^{\prime\prime})\rightarrow Map_{/E}(E^{\prime\prime}\times_{B^{\prime\prime}}B,E^\prime\times_{B^\prime}B),\ \Phi\mapsto\varphi_1^*(\Phi^\sharp).
		\end{equation} Combining all of the above, we finally have constructed a functor \begin{equation}
		    X:\Delta_+\cong(\tilde{\Delta}_+)_{[1]/}^{op}\xrightarrow{\tilde{X}^{op}}(\P)^{op}_{p/}\xrightarrow{g}\S_{/E}\xrightarrow{forget}\S.
		\end{equation}Unwinding the notions, we see that $X$ is the augmented cosimplicial space as below \begin{equation}\label{idieh}
		    B\underset{d_0}{\rightarrow}E\overset{s_0}{\underset{d_0,d_1}\leftrightarrows}B\times_{B_2}E_2\overset{s_0,s_1}{\underset{d_0,d_1,d_2}\leftrightarrows}B\times_{B_3}E_3\overset{s_0,s_1,s_2}{\underset{d_0,d_1,d_2,d_3}\leftrightarrows}B\times_{B_4}E_4\rightleftarrows\cdots.
		\end{equation} Here, the first $d_0$ is induced by the counit $\epsilon$ as a global section.
    \item In the second part, we will prove the Segal-like condition. We first define another morphism $q:E\rightarrow B$ besides $p$. In \cref{augsump}, because $s_0^0\circ d_1^1=id_{[0]}$, $p\circ p\xrightarrow{d_1}p\circ id_{\S}\simeq p$ is a retract with respect to $\delta:p\rightarrow p\circ p$. It follows that \begin{equation}
        \delta_1:B\rightarrow B_2\simeq\colim_BMap_{\S}(p[b],B)
    \end{equation} is a section. By \cref{repeat2}, we can view $\delta_1$ as a point in \begin{equation}
        \lim_BMap_{\S}(p[b],B)\simeq Map_{\S}(\colim_Bp[b],B)\simeq Map_{\S}(E,B)
    \end{equation} which we denote by $q$\footnote{Here if we switch to \cref{reviewdef}, the map $q:E\rightarrow B$ can also be characterized as follows: because $B_2\simeq p(B)\simeq p_*(B\times E)$, under the adjunction $p^*\dashv p_*$, the map $\delta_1:B\rightarrow p_*(B\times E)$ over $B$ is sent to a map $p^*(B)\simeq E\rightarrow B\times E$ over $E$, which is just a map $q:E\rightarrow B$.}. Next, in the augmented cosimplicial space $X$ constructed in the first part, we will prove that \begin{equation}
        X_{n-1}\simeq B\times_{B_n}E_n\simeq\lim(\underbrace{E\xrightarrow{q}B\xleftarrow{p}E\xrightarrow{q}B\cdots\xleftarrow{p}E}_{(n-1)\textrm{ copies of }q\textrm{ and }p})
    \end{equation} by induction on $n\in\mathbb{N}$. The cases $n=0,1$ are vacuously true. It remains to show that \begin{equation}
        X_n\simeq E\times_BX_{n-1}\simeq\lim(E\xrightarrow{q}B\xleftarrow{f_n^*(p^{\circ n})}X_{n-1})
    \end{equation} for every $n\geq2$, where $f_n:B\rightarrow B_n$ is the positional morphism of $p\rightarrow p^{\circ n}$ in \cref{augsump}. But by the pasting law, we are reduced to showing that \begin{equation}\label{yierer}
        X_n\simeq E\times_B(B\times_{B_n}E_n)\simeq E\times_{B_n}E_n\simeq\lim(E\xrightarrow{f_n\circ q}B_n\xleftarrow{p^{\circ n}}E_n).
    \end{equation} To this end, we first prove by an easy induction that \begin{equation}
        f_{n+1}:B\rightarrow B_{n+1}\simeq\colim_BMap_{\S}(p[b],B_n),\ \forall n\in\mathbb{N}^*
    \end{equation}is a section given by \begin{equation}\label{chchch}
        f_{n+1}:b\mapsto(p[b]\xrightarrow{f_n\circ\delta_1(b)}B_n).
    \end{equation} The base case $n=1$ has already been proved since $f_2=\delta_1$. For the induction step, we note that $f_{n+1}$ can be written as a composition \begin{equation}
      B\xrightarrow{\delta_1}B_2\simeq\colim_BMap_{\S}(p[b],B)\xrightarrow{\colim_B(f_n\circ-)}B_{n+1}\simeq\colim_BMap_{\S}(p[b],B_n)
    \end{equation} where the second arrow is the map $p(f_n):B_2\simeq p(B)\rightarrow B_{n+1}\simeq p(B_n)$. Then it is clear that the case of $n+1$ follows from the case of $n$, $\forall n\in\mathbb{N}^*$.\par
    Return to the proof of \cref{yierer}. Since pullbacks in $\S$ preserve colimit, \begin{equation}
      E\times_{B_n}E_n\simeq\colim(E\xrightarrow{q}B\xrightarrow{f_n}B_n\xrightarrow{p^{\circ n}}\S). 
    \end{equation}Recall from the construction of $q$, that $q|_{p[b]}=\delta_1(b)$, $\forall b\in B$. Then by \cref{chchch}, the diagram above over which we take the colimit restricted on the fiber $p[b]$ becomes \begin{equation}
        p[b]\xrightarrow{f_{n+1}(b)}B_n\xrightarrow{p^{\circ n}}\S,\ \ \ \forall b\in B.
    \end{equation} On the other hand, we have \begin{equation}
        X_n\simeq B\times_{B_{n+1}}E_{n+1}\simeq\colim(B\xrightarrow{f_{n+1}}B_{n+1}\xrightarrow{p^{\circ(n+1)}}\S).
    \end{equation} Since $p^{\circ(n+1)}\simeq p\circ p^{\circ n}$, recall from the explicit description of polynomial composition, that \begin{equation}
        p^{\circ(n+1)}[f_{n+1}(b)]\simeq\colim_{p[b]}(p^{\circ n}\circ f_{n+1}(b)),\ \ \ \forall b\in B.
    \end{equation}Comparing the above four equations, we immediately deduce that $X_n\simeq E\times_{B_n}E_n$ from \cref{dbcolim1}.
    \end{enumerate}We have thus completed the proof of our main theorem of this section.
\end{proof}
\begin{rmk}
    Let us conclude this paper by commenting on a possible way to improve the result. So we have constructed a functor $\Delta_+\rightarrow\S$ in \cref{idieh}, the remaining goal is to extend it to a functor $\Delta^{op}\rightarrow\S$. This can not be done via a Kan extension along $e_+$. Instead, a promising argument goes like the following: $\Delta^{op}$ has two wide subcategories, one is $e_+(\Delta_+)$, and the other one is spanned by all convex morphisms. It should not be hard to construct a functor on the latter subcategory to $\S$, also sending $[n]$ to $\underbrace{E\times_BE\cdots \times_BE}_{n\textrm{ copies of }E}$. Then the key observation is that the two wide subcategories here form an \textit{orthogonal factorization system} (\cite[section 9.3.8]{Kerodon}) of $\Delta^{op}$, and there should somehow be a way to patch these two functors into one on $\Delta^{op}$, as desired. However, this idea requires further precision.
\end{rmk}

\bibliographystyle{alpha}
\bibliography{main}

\begin{thebibliography}{GHK22}

\bibitem[AU16]{mainthm}
Danel Ahman and Tarmo Uustalu.
\newblock {\em Directed Containers as Categories}.
\newblock volume 207, pages 89--98, 2016.

\bibitem[BGN18]{barwick2018dualizing}
Clark Barwick, Saul Glasman, and Denis Nardin.
\newblock {\em Dualizing cartesian and cocartesian fibrations}.
\newblock {\em Theory Appl. Categ}, 33(4):67--94, 2018.

\bibitem[BS16]{clark}
Clark Barwick and Jay Shah.
\newblock {\em Fibrations in infinity-category theory}.
\newblock {\em preprint arXiv:math/1607.04343v3}, 2016.

\bibitem[GHK22]{gepner2022operads}
David Gepner, Rune Haugseng, and Joachim Kock.
\newblock {\em Infinity-operads as analytic monads}.
\newblock {\em International Mathematics Research Notices},
  2022(16):12516--12624, 2022.

\bibitem[GK13]{gambino2013polynomial}
Nicola Gambino and Joachim Kock.
\newblock {\em Polynomial functors and polynomial monads}.
\newblock In {\em Mathematical proceedings of the cambridge philosophical
  society}, volume 154, pages 153--192. Cambridge University Press, 2013.

\bibitem[Ill06]{illusie2006complexe}
Luc Illusie.
\newblock {\em {\em Complexe cotangent et d{\'e}formations I}}, volume 239.
\newblock Springer, 2006.

\bibitem[Lur07]{feijiaohuanl}
Jacob Lurie.
\newblock {\em Derived algebraic geometry II: noncommutative algebra}.
\newblock {\em preprint arXiv:math/0702299v5}, 2007.

\bibitem[Lur09a]{arxivagain}
Jacob Lurie.
\newblock {\em $(\infty,2)$-categories and the Goodwillie calculus I}.
\newblock {\em preprint arXiv:math/0905.0462v2}, 2009.

\bibitem[Lur09b]{HTT}
Jacob Lurie.
\newblock {\em {\em Higher Topos Theory}, volume 170 of {\em
  \MakeUppercase{a}nnals of \MakeUppercase{m}athematics
  \MakeUppercase{s}tudies}}.
\newblock Princeton University Press, Princeton, NJ, 2009.

\bibitem[Lur17]{HA}
Jacob Lurie.
\newblock {\em Higher Algebra}.
\newblock {\em http://www.math.harvard.edu/~lurie/papers/HA.pdf}, 2017.

\bibitem[Lur25]{Kerodon}
Jacob Lurie.
\newblock {\em Kerodon}.
\newblock {\em http://kerodon.net/kerodon.pdf}, 2025.

\bibitem[NS23]{niu2023polynomial}
Nelson Niu and David~I Spivak.
\newblock {\em Polynomial functors: A mathematical theory of interaction}.
\newblock {\em preprint arXiv:math/2312.00990v2}, 2023.

\bibitem[Rez01]{rezk2001model}
Charles Rezk.
\newblock {\em A model for the homotopy theory of homotopy theory}.
\newblock {\em Transactions of the American Mathematical Society},
  353(3):973--1007, 2001.

\bibitem[SGF25]{spivak2025functorial}
David~I Spivak, Richard Garner, and Aaron~David Fairbanks.
\newblock {\em Functorial aggregation}.
\newblock {\em Journal of Pure and Applied Algebra}, 229(2):107883, 2025.

\bibitem[Spi20]{dnmc}
David~I Spivak.
\newblock {\em Poly: An abundant categorical setting for mode-dependent
  dynamics}.
\newblock {\em preprint arXiv:math/2005.01894v2}, 2020.

\bibitem[Spi21]{lang}
David~I Spivak.
\newblock {\em Learners' languages}.
\newblock In {\em Proceedings of the 4th International Conference on Applied
  Category Theory}, pages 14--28, 2021.

\bibitem[Spi22a]{handbook}
David~I Spivak.
\newblock {\em A summary of categorical structures in Poly}.
\newblock {\em preprint arXiv:math/2202.00534v14}, 2022.

\bibitem[Spi22b]{shannon}
David~I Spivak.
\newblock {\em Polynomial functors and Shannon entropy}.
\newblock In {\em Proceedings of the 5th International Conference on Applied
  Category Theory}, pages 331--343, 2022.

\bibitem[SS23]{neizaifanchou}
Brandon~T Shapiro and David~I Spivak.
\newblock {\em Structures on categories of polynomials}.
\newblock {\em preprint arXiv:math/2305.00167v2}, 2023.

\bibitem[Tor25]{torii2025duoidal}
Takeshi Torii.
\newblock {\em On duoidal infinity-categories}.
\newblock {\em Journal of Homotopy and Related Structures}, 20(1):125--162,
  2025.

\bibitem[Web15]{weber2015polynomials}
Mark Weber.
\newblock {\em Polynomials in categories with pullbacks}.
\newblock {\em Theory Appl. Categ}, 30(16):533--598, 2015.

\end{thebibliography}
\end{document}